\documentclass{amsart}
\usepackage{amsmath}
\usepackage{amsfonts}
\usepackage{amssymb}
\usepackage{amsrefs}
\usepackage{amsthm}
\usepackage{array}

\newtheorem*{thm}{Theorem}
\newtheorem{thmn}{Theorem}
\newtheorem{prop}{Proposition}
\newtheorem*{lem}{Lemma}

\newtheorem*{cor}{Corollary}
\numberwithin{equation}{section}

\begin{document}
\title{Exponential Hilbert series of equivariant embeddings}
\author{Wayne A. Johnson}
\address{University of Wisconsin-Platteville\\
		1 University Plaza\\
		Platteville, WI 53818}
\email{johnsonway@uwplatt.edu}
\begin{abstract}
In this article, we study properties of the exponential Hilbert series of a $G$-equivariant projective variety, where $G$ is a semisimple, simply-connected complex linear algebraic group. We prove a relationship between the exponential Hilbert series and the degree and dimension of the variety. We then prove a combinatorial identity for the coefficients of the polynomial representing the exponential Hilbert series. This formula is used in examples to prove further combinatorial identities involving Stirling numbers of the first and second kinds.
\end{abstract}
\maketitle

\section{Preliminaries}

To begin with, we set our notation and collect previous results that are needed in the paper. The main results preceeding those in the paper can be found in more detail in \cite{GrW} and \cite{Johnson2}. One of the main results of \cite{Johnson2} (see \S1.3 below) is that the exponential Hilbert series of an equivariant embedding of a flag variety converges to the product of a rational polynomial, $p(x)$, and an exponential term. The present article is concerned with computing the coefficients found in $p(x)$. In \S2, we present geometric data that is encoded in the coefficient and degree of the leading term of $p(x)$; in \S3, we prove a combinatorial formula for the coefficients in degrees $1< k< d$, where $d$ is the degree of $p(x)$. This, coupled with Theorem 4 of \cite{Johnson2}, provides a complete formulation of the coefficients of $p(x)$. This formulation is then used to prove combinatorial identities involving the Stirling numbers of the second kind, including the well-known recursive relationship
\begin{center}
$S(n+1,k+1)=\displaystyle\sum_{j=k}^n{n\choose j}S(j,k)$
\end{center}
using the representation theory of a complex linear algebraic group. For the remainder of this section, we review some materials needed for the proofs in \S2, \S3.

\subsection{Equivariant embeddings}

Throughout the paper, we assume $G$ is a semisimple, simply-connected linear algebraic group over $\mathbb{C}$. Let $\mathfrak{g}=\text{Lie}(G)$. Fixing a Borel subgroup, $B$, and maximal torus, $T\subseteq B$, with corresponding algebras $\mathfrak{b}=\text{Lie}(B)$ and $\mathfrak{h}=\text{Lie}(T)$, determines a set of positive roots, $\Phi^+$, corresponding to the pair $(\mathfrak{b},\mathfrak{h})$. Denote the set of dominant weights corresponding to $(\mathfrak{b},\mathfrak{h})$ by $P_+(\mathfrak{g})$.

Let $\lambda\in P_+(\mathfrak{g})$. Then $\lambda$ corresponds to a parabolic subgroup, $P_\lambda\supseteq B$. To construct $P_\lambda$, let $L(\lambda)$ denote the finite-dimensional irreducible representation of $G$ of highest weight $\lambda$, and let $S(\lambda)$ denote its dual representation. Then there exists a unique functional, $f_\lambda\in S(\lambda)$, fixed by $B$. We then define $P_\lambda$ to be the stabilizer of $f_\lambda$ in $G$. Equivalently, $P_\lambda$ is the stabilizer of the unique hyperplane, $H_\lambda$, in $L(\lambda)$ annihilated by $f_\lambda$. This latter characterization is the one we take.

The parabolic subgroup $P_\lambda$ corresponds to a $G$-equivariant projective variety, $X_\lambda$. To construct $X_\lambda$, let $\pi_\lambda:G/P_\lambda\rightarrow\mathbb{P}(L(\lambda))$ be the map
\begin{center}
$\pi_\lambda:gP_\lambda\mapsto g.H_\lambda$
\end{center}
from the quotient group $G/P_\lambda$ to the projective space $\mathbb{P}(L(\lambda))$ of hyperplanes in $L(\lambda)$. The image of $G/P_\lambda$, which we denote $X_\lambda$, is the unique closed orbit of $G$ on $\mathbb{P}(L(\lambda))$, and as such, is a homogeneous, non-singular projective variety (see, \cite{FuH}, \cite{GrW}). Any $G$-equivariant projective variety can be realized as an equivariant embedding $\pi_\lambda(G/P_\lambda)=X_\lambda$.

Throughout the text, if we wish to discuss the projective variety on its own, we refer to it as $X_\lambda$; if we wish to discuss its particular embedding in $\mathbb{P}(L(\lambda))$, we refer to it as $\pi_\lambda$. As a projective variety, $X_\lambda$ has  a graded homogeneous coordinate ring $\mathbb{C}[X_\lambda]$. In \cite{GrW}, it is shown that
\begin{center}
$\mathbb{C}[X_\lambda]=\displaystyle\bigoplus_{n\geq0}L(n\lambda)$,
\end{center}
both as a $G$-representation and as a graded ring.

\subsection{Hilbert series and Hilbert polynomial}

If $A$ is a graded $\mathbb{C}$-algebra with homogeneous components $A_n$, the \emph{Hilbert series} of $A$ is the formal power series
\begin{center}
$HS_A(x):=\displaystyle\sum_{n\geq0}\dim(A_n)x^n$.
\end{center}
The function $HF_A(n)=\dim(A_n)$ is called the \emph{Hilbert function} of $A$. Asymptotically, the Hilbert function is polynomial. This means that, for $n$ large enough, $HF_A(n)=p(n)$ for some polynomial $p(t)$, called the \emph{Hilbert polynomial} of $A$. When $A$ is the homogeneous coordinate ring of a projective variety, some geometric data is encoded in the Hilbert series and polynomial. In particular, if $A$ is finitely-generated, then
\begin{center}
$HS_A(x)=\displaystyle\frac{q(x)}{(1-x)^{d+1}}$,
\end{center}
where $q(x)$ is an integer polynomial such that $q(1)$ is the \emph{degree} of the embedding of the projective variety and $d$ is the \emph{dimension} of the variety (see \cite{AtM}). In terms of the Hilbert polynomial, we have
\begin{center}
$p(t)=a\displaystyle\frac{x^d}{d!}+\text{lower order terms}$,
\end{center}
where $a$ is the degree of the embedding of the projective variety and $d$ is again the dimension of the variety (see \cite{GrW}, \cite{MaV}). Computation of the Hilbert polynomial and/or the Hilbert series of $A$ then yields geometric information about the variety, which can be computed in a very simple way.

In the case of the embedding $\pi_\lambda$, formulas for the Hilbert series and Hilbert polynomial are given in \cite{GrW}. We collect the main statements below.

\begin{thm}
The Hilbert series of $\pi_\lambda$ is given by
\begin{equation}
HS_\lambda(x)=\displaystyle\prod_{\alpha\in\Phi^+}\left(1+c_\lambda(\alpha)x\frac{d}{dx}\right)\frac{1}{1-x},
\end{equation}
where $c_\lambda(x):=\displaystyle\frac{(\lambda,\alpha)}{(\rho,\alpha)}$ for the nondegenerate, bilinear form $(\cdot,\cdot)$ on $\mathfrak{h}^*$ induced by the Killing form on $\mathfrak{g}$. Further, the Hilbert polynomial of $\pi_\lambda$ is given by
\begin{equation}
H_\lambda(t)=\displaystyle\prod_{\alpha\in\Phi^+}(1+t\cdot c_{\lambda}(\alpha)).
\end{equation}
\end{thm}

Note that, in (2.1), the product over the positive roots is a differential operator that is being applied to a rational function. We will denote a differential operator acting on an expression by juxtaposition throughout the paper.

\subsection{Exponential Hilbert series}

In \cite{Johnson2}, the author presents analogous results to those in \cite{GrW} and \cite{Johnson} for an exponential generating function called an \emph{exponential Hilbert series}. This formal power series is similar to $HS_A(x)$. Given a graded algebra, $A$, with graded components, $A_n$, we define its exponential Hilbert series to be the formal power series
\begin{center}
$E_A(x)=\displaystyle\sum_{n\geq0}\dim(A_n)\frac{x^n}{n!}$.
\end{center}
This paper is concerned with a further exploration of $E_A(x)$, in the context of the equivariant embeddings of \cite{GrW}. In \cite{Johnson2}, analogous formulas to those from \cite{GrW} are proved. For example, let $E_\lambda(x)$ denote the exponential Hilbert series of the embedding $\pi_\lambda$. In other words,
\begin{equation}
E_\lambda(x)=\displaystyle\sum_{n\geq0}\dim(L(\lambda))\frac{x^n}{n!}.
\end{equation}
In \cite{Johnson2}, it is shown that
\begin{equation}
E_\lambda(x)=\displaystyle\prod_{\alpha\in\Phi^+}\left(1+c_\lambda(\alpha)x\frac{d}{dx}\right)e^x,
\end{equation}
similarly to the results from \cite{GrW}. It is further shown that this formula gives $E_\lambda(x)$ in the form $p(x)e^x$ for some polynomial $p(x)$ with rational coefficients. The author goes on to prove that the following result (see \cite{Johnson2}, Theorem 4).
\begin{thm}
The constant term and linear coefficient of $p(x)$ sum to $\dim(L(\lambda))$.
\end{thm}
This paper is mainly concerned with computing the remaining coefficients of $p(x)$. These coefficients have interesting representation-theoretic and combinatorial properties. In \S3, we show that the degree of $\pi_\lambda$ and the dimension of $X_\lambda$ can be read off of the coefficients of $p(x)$. From this point of view, the polynomial $p(x)$ replaces the Hilbert polynomial in the exponential case. In \S4, we then give a combinatorial description of the remaining coefficients of $p(x)$, generalizing Theorem 4 in \cite{Johnson2}. This formula is given in terms of the constants $c_\lambda(\alpha)$ and the Stirling numbers of the second kind, which we review below.

\subsection{Stirling numbers of the second kind}

The Stirling numbers of the second kind\footnote{See integer sequence A008277 of the Online Encyclopedia of Integer Sequences (OEIS).} appear in \S4 in our computation of the coefficients of $p(x)$. We recall their definition and some properties below.

We denote a Stirling number of the second kind as $S(n,k)$, where $S(n,k)$ is the number of ways to partition an $n$-element set into $k$ nonempty subsets. For example, the set $\{1,2,3\}$ can be partitioned into $\{1\}\cup\{2,3\}$, $\{2\}\cup\{1,3\}$, and $\{3\}\cup\{1,2\}$. Therefore, $S(3,2)=3$. These numbers can be arranged into a triangle. The first few rows appear below.
\begin{center}
\begin{tabular}{>{$n=}l<{$\hspace{12pt}}*{13}{c}}
1 &&&&&&&1&&&&&&\\
2 &&&&&&1&&1&&&&&\\
3 &&&&&1&&3&&1&&&&\\
4 &&&&1&&7&&6&&1&&&\\
5 &&&1&&15&&25&&10&&1&&
\end{tabular}
\vspace{0.5cm}
\\
\end{center}
It is sometimes convenient to assume that $S(0,0)=1$ and $S(n,0)=0$ for all $n>0$. These numbers often appear in formulas for exponential generating functions, and a formula for $E_\lambda(x)$ in terms of them appears in \cite{Johnson2}. We will use them in a related way in \S3.

\section{The degree and dimension of $X_\lambda$}

As before, we let $X_\lambda$ denote the $G$-equivariant projective variety corresponding to the dominant integral weight $\lambda\in P_+(\mathfrak{g})$. In this section, we prove the formulas for the dimension of $X_\lambda$ and the degree of the embedding $\pi_\lambda$ of $X_\lambda$ in $\mathbb{P}(V(\lambda))$. Recall that the exponential Hilbert series $E_\lambda(x)$ converges to a product $p(x)e^x$, where $p(x)$ is a polynomial with rational coefficients. Throughout the rest of the paper, $p(x)$ will always refer to this polynomial. We begin with a Lemma.

\begin{lem}
Let $d$ be the number of positive roots $\alpha$ such that $c_\lambda(\alpha)\neq0$. Then $d=\deg(p(x))$. 
\end{lem}

\begin{proof}
Each nonzero $c_\lambda(\alpha)$ corresponds to a differential operator contributing to the computation of $E_\lambda(x)$. We claim that applying the differential operator 
\begin{equation}
\left(1+c_\lambda(\alpha)x\displaystyle\frac{d}{dx}\right)
\end{equation}
to $q(x)e^x$ raises the degree of $q(x)$ by one for any polynomial $q(x)$. Applying (2.1) to $q(x)e^x$, we get
\begin{center}
$q(x)e^x+ax[q'(x)e^x+q(x)e^x]$.
\end{center}
As $ax\cdot q'(x)$ has the same degree as $q(x)$ and $ax\cdot q(x)$ has degree one higher, the claim holds, and so does the Lemma.
\end{proof}

This lemma will be used in our computation of the dimension of $X_\lambda$ below.

\begin{thmn}
The dimension of $X_\lambda$ is the degree of $p(x)$.
\end{thmn}

\begin{proof}
Let $H_\lambda(x)$ denote the Hilbert polynomial of $X_\lambda$. Then 
\begin{center}
$\dim(X_\lambda)=\deg(H_\lambda(x))$. 
\end{center}
In \cite{GrW}, it is shown that
\begin{equation}
H_\lambda(t)=\displaystyle\prod_{\alpha\in\Phi^+}(1+t\cdot c_{\lambda}(\alpha)).
\end{equation}
This polynomial has degree equal to the number of nonzero $c_\lambda(\alpha)$. Then the theorem follows from the previous lemma.
\end{proof}

We now turn to computing the degree of the equivariant embedding $\pi_\lambda$ of $X_\lambda$, which will follow as a corollary to the next result.

\begin{thmn}
The leading coefficient of $p(x)$ is $\displaystyle\prod_{\alpha\in\Phi^+}c_\lambda(\alpha)$.
\end{thmn}

\begin{proof}
We prove something slightly more general: consider the product
\begin{center}
$\displaystyle\prod_{i=1}^d\left(1+a_ix\frac{d}{dx}\right)e^x$,
\end{center}
for any constants $a_i$. If $E_j$ denotes the $j^{th}$ elementary symmetric function in $d $ variables, we have
\begin{center}
$\displaystyle\prod_{i=0}^d\left(1+a_ix\frac{d}{dx}\right)e^x=\left[\sum_{j=0}^dE_j\left(a_1,\dots,a_d\right)\left(x\frac{d}{dx}\right)^j\right]e^x$.
\end{center} 
The product
\begin{center}
$\displaystyle\left(x\frac{d}{dx}\right)^je^x$
\end{center}
can be written in terms of the Stirling polynomials. Namely,
\begin{center}
$\displaystyle\left(x\frac{d}{dx}\right)^je^x=\phi_j(x)e^x$.
\end{center}
As these polynomials increase in degree as $j$ increases, the highest degree term occurs when $j=d$, and therefore, because the leading coefficient of $\phi_j(x)$ is always one, the coefficient of the leading term in $p(x)$ is
\begin{center}
$\displaystyle\prod_{j=1}^da_j=E_d(a_1,\dots,a_d)$.
\end{center} 
As this holds, for any constants, it certainly holds for
\begin{center}
$E_\lambda(x)=\displaystyle\prod_{\alpha\in\Phi^+}\left(1+c_\lambda(\alpha)x\frac{d}{dx}\right)e^x$,
\end{center}
and the theorem follows.
\end{proof}

Combining the above result with a result from \cite{GrW}, we can prove the following corollary.

\begin{cor}
If $a_d$ is the leading coefficient of $p(x)$, then
\begin{center}
$\deg(\pi_\lambda)=d!\cdot a_d$.
\end{center}
\end{cor}

\begin{proof}
This follows immediately from the previous theorem coupled with the following result from \cite{GrW}:
\begin{center}
$\deg(\pi_\lambda)=d!\displaystyle\prod_{\alpha\in\Phi^+}c_\lambda(\alpha)$.
\end{center}
\end{proof}

At this point, we sum up what has been shown about the terms of the polynomial $p(x)$. In \cite{Johnson2}, it was shown that the linear term of $p(x)$ has coefficent $\dim(L(\lambda))-1$ and that the constant term of $p(x)$ will always be one. This gives
\begin{center}
$p(x)=1+(\dim(L(\lambda)-1)x+\dots+\displaystyle\frac{\deg(\pi_\lambda)}{d!}x^d$,
\end{center}
where $d=\dim(X_\lambda)$. In the next section, we present a combinatorial description of the missing coefficients. They can be computed in terms of the constants, $c_\lambda(\alpha)$, and the Stirling numbers of the second kind.

\section{The remaining coefficients of $p(x)$}

In this section, we look at the general problem of computing all of the coefficients of $p(x)$. We develop a formula for these in terms of the constants $c_\lambda(\alpha)$ and the Stirling numbers of the second kind. At present, this is a purely combinatorial formula. It would be interesting to find some Lie theoretic significance for these coefficients.

We return to the exponential Hilbert series $E_\lambda(\alpha)=p(x)e^x$. We let $\Phi^+=\{\beta_1,\dots,\beta_d\}$ be an enumeration of the positive roots of $\mathfrak{g}$ and set $a_i=c_\lambda(\beta_i)$ for each positive root $\beta_i$. Then,
\begin{center}
$E_\lambda(x)=\displaystyle\left[\sum_{j=0}^dE_j\left(a_1,\dots,a_d\right)\left(x\frac{d}{dx}\right)^j\right]e^x$,
\end{center}
as before, where $E_j$ denotes the $j^{th}$ elementary symmetric function in $d$ variables. As $E_0(a_1,\dots,a_d)=1$, we can rewrite this as
\begin{center}
$E_\lambda(x)=\displaystyle\left[1+\sum_{j=1}^dE_j\left(a_1,\dots,a_d\right)\left(x\frac{d}{dx}\right)^j\right]e^x$.
\end{center}
For $j\geq1$, we have
\begin{center}
$\displaystyle\left(x\frac{d}{dx}\right)^je^x=\phi_j(x)e^x$,
\end{center}
where $\phi_j(x)$ is the $j^{th}$ Stirling polynomial. In other words,
\begin{center}
$\phi_j(x)=S(j,1)+S(j,2)x+\dots+S(j,j)x^j$,
\end{center}
where $S(j,i)$ is the number of ways to partition a $j$ element set into $i$ nonempty subsets. Recall, these numbers are called the Stirling numbers of the second kind.

Then we have
\begin{center}
$E_\lambda(x)=\left[1+\displaystyle\sum_{j=1}^dE_j(a_1,\dots,a_d)\sum_{i=1}^jS(j,i)x^i\right]e^x$.
\end{center}
The term in brackets is then a formula for the polynomial $p(x)$, namely
\begin{equation}
p(x)=1+\displaystyle\sum_{j=1}^dE_j(a_1,\dots,a_d)\sum_{i=1}^jS(j,i)x^i.
\end{equation}
From (3.1), it is immediately obvious that the constant term of $p(x)$ is one. As we saw in \S3, the leading term is the product of the coefficients, $a_i$. In Theorem 4 of \cite{Johnson2}, the author shows that the coefficient of $x$ is $\dim(L(\lambda))-1$. We compute the remaining coefficients below.

\begin{thmn}
Let $1\leq k\leq d=|\Phi^+|$. The coefficient of $x^k$ in $p(x)$ is given by
\begin{center}
$\displaystyle\sum_{j=k}^dE_j(a_1,\dots,a_d)S(j,k)$,
\end{center}
where $E_j$ is the $j^{th}$ elementary symmetric function in $d$ variables and $S(j,k)$ is the $k^{th}$ entry of row $j$ of the triangle of Stirling numbers of the second kind.
\end{thmn}

\begin{proof}
Assume $1\leq k\leq d$. The only summands in (3.1) which contain a degree $k$ term are those with $j\geq k$. Therefore, the degree $k$ term of $p(x)$ is the same as the degree $k$ term of 
\begin{center}
$q(x):=\displaystyle\sum_{j=k}^dE_j(a_1,\dots,a_d)\sum_{i=1}^jS(j,i)x^i$.
\end{center}
The degree $k$ terms occur when $i=k$, and there is one for each elementary symmetric function in $q(x)$. Therefore, as claimed, the coefficient of the degree $k$ term is
\begin{center}
$\displaystyle\sum_{j=k}^dE_j(a_1,\dots,a_d)S(j,k)$.
\end{center}
\end{proof}

Note that Theorem 3 immediately implies that the coefficient of $x$ in $p(x)$ is $\dim(L(\lambda))-1$, as was shown in Theorem 4 of \cite{Johnson2}. As the Theorem still holds when $k=d$, Theorem 3 is a direct generalization of Theorem 2 in \S3.

\section{Examples}

In this section, we collect some examples of both representation-theoretic and combinatorial interest. The first example is in type $A_n$ and showcases how, through Theorem 3, we can prove interesting combinatorial formulas using the representation theory of $G$. These examples are not meant to be exhaustive. Instead, they are chosen to illustrate the types of combinatorial identities that can be proved using $E_\lambda(x)$.

\subsection{The first fundamental representation in type $A_n$}

In \cite{Johnson2}, a formula for $E_{\lambda}(x)$ was given for $\lambda=\omega_1$, where $\omega_1$ is the first fundamental dominant weight in type $A_n$. Here, we use Theorem 3 to explicitly compute the coefficients of $E_{\omega_1}(x)$. This leads to a combinatorial formula relating symmetric functions, binomial coefficients, and the Stirling numbers of the second kind, which leads to other interesting combinatorial identities.

\begin{prop}
Let $G=SL(n+1,\mathbb{C})$ $(n\geq1)$, and let $\omega_1$ be the first fundamental dominant weight of $G$. Then
\begin{equation}
E_{\omega_1}(x)=p(x)e^x,
\end{equation}
where $p(x)$ is the polynomial of the form
\begin{equation}
p(x)=\displaystyle\sum_{j=0}^n\frac{1}{j!}{n\choose j}x^j,
\end{equation}
\end{prop}

\begin{proof}
The proof is by induction on the rank of $G$. For this induction, we use a subscript to denote the rank of the group: $E_{\omega_1}^n(x)$ denotes the exponential Hilbert series in rank $n$. The base case is when $n=1$, where we have
\begin{center}
$E_{\omega_1}^1(x)=(1+x)e^x$,
\end{center}
with coefficients $1=\frac{1}{0!}$ and $1=\frac{1}{1!}$. As the second row of Pascal's triangle consists of two ones, the claim holds when $n=1$.

Now assume the claim holds for rank $n-1$ and consider $E_{\omega_1}^n$. By Proposition 1 in \cite{Johnson2} we have
\begin{center}
$E_{\omega_1}^n(x)=\displaystyle\prod_{j=0}^n\left(1+\frac{x}{n}\frac{d}{dx}\right)e^x$.
\end{center}
Therefore, we can write $E_{\omega_1}^n(x)$ recursively:
\begin{center}
$E_{\omega_1}^n(x)=\left(1+\displaystyle\frac{x}{n}\frac{d}{dx}\right)E_{\omega_1}^{n-1}(x)$.
\end{center}
Since we assume the claim holds for rank $n-1$, and by using the product rule, we can then write
\begin{center}
$E_{\omega_1}^n(x)=e^x\left[\displaystyle\sum_{j=0}^{n-1}\left[\left(\frac{b_j}{j!}+\frac{jb_j}{nj!}\right)x^j+\frac{b_j}{nj!}x^{j+1}\right]\right]$,
\end{center}
where $b_j$ is the corresponding entry of Pascal's Triangle in row $n-1$.

We finish the proof by computing the coefficients of the various powers of $x$. First, for the constant term, the coefficient is
\begin{center}
$\displaystyle\frac{b_0}{0!}+\frac{0b_0}{n0!}=\frac{1}{1}=\frac{1}{0!}$,
\end{center}
as we expect. The coefficient of $x^n$ is also simply computed:
\begin{center}
$\displaystyle\frac{b_{n-1}}{n(n-1)!}=\frac{1}{n!}=\frac{1}{n!}$.
\end{center}
The computation of the other coefficients is a little more involved. Choose $k$ such that $1\leq k\leq n-1$. Then the coefficient of $x^k$ is
\begin{center}
$\displaystyle\frac{b_k}{k!}+\frac{kb_k}{nk!}+\frac{b_{k-1}}{n(k-1)!}=\frac{nb_k+kb_k+kb_{k-1}}{nk!}$.
\end{center}
If we convert $b_k$ and $b_{k-1}$ into binomial coefficents, the coefficient becomes
\begin{center}
$\displaystyle\left(\frac{n(n-1)!}{k!(n-1-k)!}+\frac{k(n-1)!}{k!(n-1-k)!}+\frac{k(n-1)!}{(k-1)!(n-k)!}\right)\frac{1}{nk!}$,
\end{center}
which can be rearranged and then simplified to yield
\begin{center}
$\displaystyle\frac{n!}{k!}\frac{1}{k!(n-k)!}=\frac{{n\choose k}}{k!}$,
\end{center}
and the claim is proved.
\end{proof}

Note that the above proof is dependent on the fact that, for $\omega_1$, we have
\begin{center}
$c_{\omega_1}(\alpha)=\displaystyle\frac{1}{l(\alpha)}$,
\end{center}
where $l(\alpha)$ denotes the length of the positive root $\alpha$, if the simple root, $\alpha_1$, dual to $\omega_1$ appears in $\alpha$, and is zero otherwise. Let $\Phi^+=\{\beta_1\dots,\beta_d\}$ and $d=|\Phi^+|$, as before. As many of the $c_{\omega_1}(\alpha)$ terms are zero, and those that are not zero give values of $1, 1/2, 1/3,\dots, 1/n$ exactly once each, we have
\begin{center}
$E_j(c_{\omega_1}(\beta_1),\dots,c_{\omega_1}(\beta_d))=E_j(1,1/2,\dots,1/n)$,
\end{center}
where the functions on the right are elementary symmetric functions in $n$ variables. Then from Theorem 3, we have
\begin{equation}
\displaystyle\sum_{j=k}^nE_j(1,1/2,\dots,1/n)S(j,k)=\frac{{n\choose k}}{k!},
\end{equation}
which is a purely combinatorial formula relating the Stirling numbers of the second kind and the binomial coefficients using representation theory.

We can go further. There is a well-known relationship between the elementary symmetric functions and a sequence of numbers called \emph{the Stirling numbers of the first kind}, which we will denote $c(n,k)$ (here we assume these numbers to be unsigned). Algebraically, $c(n,k)$ counts the number of permutations of $n$ elements with $k$ disjoint cycles (see Corollary 12.1 in \cite{Cha}). Explicitly,
\begin{center}
$E_j(1,1/2,\dots,1/n)=\displaystyle\frac{c(n+1,n+2-j)}{n!}$.
\end{center}
Replacing $E_j(1,1/2,\dots,1/n)$ in (4.3) then yields the following identity upon rearrangement:
\begin{equation}
\displaystyle\sum_{j=k}^nc(n+1,n+2-j)S(j,k)=\frac{1}{k!}{n\choose k}=(n-k)!{n\choose k}^2.
\end{equation}
The numbers on the right-hand side have combinatorial meaning. In particular, we have
\begin{center}
$T(n+1,n+1-k)=\displaystyle(n-k)!{n\choose k}^2$,
\end{center}
where $T(n,k)$ is the number of partial bijections of height $k$ of an $n$-element set\footnote{See integer sequence A144084 of the OEIS.}. Thus, we have a relationship between the two sequences of Stirling numbers and $T(n,k)$. In summary, we have shown that
\begin{equation}
T(n+1,n+1-k)=\displaystyle\sum_{j=k}^nc(n+1,n+2-j)S(j,k).
\end{equation}

\subsection{$E_\rho(x)$ and a classical recursion}

Assume $G$ is any semisimple, simply-connected linear algebraic group over $\mathbb{C}$, and, as is customary, let $\rho$ be the sum of the fundamental dominant weights of $G$. Explicity, if $G$ has rank $n$, then we set
\begin{center}
$\rho=\omega_1+\dots+\omega_n$,
\end{center}
where $\omega_1,\dots,\omega_n$ is an enumeration of the fundamental dominant weights of $G$. In this case, the exponential Hilbert series $E_\rho(x)$ is particularly simple to calculate, as we have
\begin{center}
$c_\rho(\alpha)=1$,
\end{center}
for any $\alpha\in\Phi^+$. Therefore,
\begin{center}
$E_\rho(x)=\left(1+x\displaystyle\frac{d}{dx}\right)^ne^x$.
\end{center}
It is straightforward to show that, in this case,
\begin{center}
$p(x)=S(n+1,1)+S(n+1,2)x+\dots+S(n+1,n+1)x^n$.
\end{center}
By Theorem 3, we have
\begin{center}
$S(n+1,k+1)=\displaystyle\sum_{j=k}^nE_j(1,1,\dots,1)S(j,k)$,
\end{center}
as the right-hand side of this equality must also give the coefficient of $x^k$ in $p(x)$. As $E_j(1,1,\dots,1)$ is just the number of terms in the $E_j(x_1,\dots,x_n)$, we have
\begin{center}
$E_j(1,1,\dots,1)=\displaystyle{n\choose j}$,
\end{center}
and we have the following combinatorial identity in the Stirling numbers of the second kind
\begin{equation}
S(n+1,k+1)=\displaystyle\sum_{j=k}^n{n\choose j}S(j,k).
\end{equation}
Thus, we have proved a well-known recursive relation of the Stirling numbers of the second kind (see, for example, Theorem 8.9 in \cite{Cha}) using representation theory.

\end{document}